\documentclass[12pt,final]{amsart}
\usepackage{amsmath,amssymb,amsthm,mathrsfs,color,dsfont}
 \usepackage[notref]{showkeys}
 \usepackage{graphicx}
\newtheorem{thm}{Theorem}

\newtheorem{lem}{Lemma}[section]

\newtheorem{rem}{Remark}[section]
\numberwithin{equation}{section}

\theoremstyle{definition}

\newcommand{\C}{{\mathbb C}}

\newcommand{\R}{{\mathbb R}}

\def\11{{\rm 1~\hspace{-1.4ex}l} }

\begin{document}

\title[Ground states for $p-$Choquard model]{Existence and uniqueness of ground states  for $p$ - Choquard model in 3D}

\author{Vladimir Georgiev}
\address[V.~Georgiev]{Department of Mathematics \\ University of Pisa \\ Largo Bruno Pontecorvo 5 \\ 56127 Pisa (Italy)  \\ and
\\ Faculty of Science and Engineering \\ Waseda University \\
 3-4-1, Okubo, Shinjuku-ku, Tokyo 169-8555 \\
Japan \\ and IMI--BAS, Acad.
Georgi Bonchev Str., Block 8, 1113 Sofia, Bulgaria}
\email{georgiev@dm.unipi.it}
\author{Mirko Tarulli}
\address[M.~Tarulli]{ Faculty of Applied Mathematics and Informatics, Technical University of Sofia, Kliment Ohridski Blvd. 8, 1000 Sofia, and IMI BAS, Acad. Georgi Bonchev Str., Block 8, 1113 Sofia, Bulgaria}
\email{mta@tu-sofia.bg}%
\author{George Venkov}
\address[G.~Venkov]{ Faculty of Applied Mathematics and Informatics, Technical University of Sofia, Kliment Ohridski Blvd. 8, 1000 Sofia, Bulgaria}
\email{gvenkov@tu-sofia.bg}%

\thanks{ The first author was supported in part by  INDAM, GNAMPA - Gruppo Nazionale per l'Analisi Matematica, la Probabilita e le loro Applicazioni, by Institute of Mathematics and Informatics, Bulgarian Academy of Sciences and by Top Global University Project, Waseda University.}

\begin{abstract}
We study the $p$-Choquard equation in 3-dimensional case and establish existence and uniqueness of ground states for the corresponding Weinstein functional. For proving the uniqueness of ground states, we use the radial symmetry
to transform the equation into an ordinary differential system, and applying the Pohozaev identities and Gronwall lemma we show that any two Weinstein minimizers coincide.
\end{abstract}

\subjclass[2010]{35Q51, 35Q40, 35Q55,  49S05}

\keywords{ p-Choquard equation, nonlocal nonlinearity, ground state}

\maketitle

\section{Introduction}

Ground states for the classical  Hartree-Choquard equation are minimizers of the Hamiltonian
\begin{equation}\label{eq_1}
    H(\psi)  = \frac{1}{2} \|\nabla \psi \|_{L^2(\R^3)}^2 - \frac{1}{4} D(|\psi|^2,|\psi|^2),
\end{equation}
where
$D(f,g)$ is the quadratic form associated with Coulomb  energy functional, i.e.
\begin{equation} \label{eq.Dsfg}
D(f,g) =  \int_{\R^3} I(f)(x) \overline{g(x)} dx
\end{equation}
and
\begin{equation} \label{eq.Dsfg2}
I(f)(x) = \frac{1}{4\pi}\int_{\R^3} f(y) \frac{dy}{|x-y|}
\end{equation}
 is the classical Riesz potential.
For any $p \geq 2$  one can define  a  modified $p$-Hamiltonian as follows
 \begin{equation}\label{HP_2}
    H_p(\psi)  = \frac{1}{2} \|\nabla ( \psi|\psi|^{(2-p)/p}) \|_{L^2(\R^3)}^2 - \frac{1}{2p} D(|\psi|^2,|\psi|^2).
\end{equation}

The ground states are
solutions to the constraint  minimization problem
\begin{equation}\label{eq.I1}
      \inf_{\{\psi \in H^1(\R^3);  \|\psi\|_{L^2(\R^3)}^2 = \lambda \}}H_{p}(\psi).
\end{equation}
A simple substitution
$$ \psi|\psi|^{(2-p)/p} = u $$  enables us to transform \eqref{eq.I1} into the  problem to find minimizer of
\begin{equation}\label{eq.I1a}
      \inf_{\{u \in \dot{H}^1(\R^3) \cap L^p(\R^3);  \|u\|_{L^p}^p = \lambda \}} \mathcal{H}_{p}(u),
\end{equation}
where $\dot{H}^1(\R^3)$ is the classical homogeneous Sobolev space and
 \begin{equation}\label{HP_3}
    \mathcal{H}_p(u)  = \frac{1}{2} \|\nabla u \|_{L^2(\R^3)}^2 - \frac{1}{2p} D(|u|^p,|u|^p).
\end{equation}
Standard symmetrization argument (we mean Schwartz symmetrization \cite{Lieb1}) and the Gagliardo - Nirenberg inequality
\begin{equation}\label{eq.I7}
   D(|u|^p,|u|^p) \leq C_{GN} \| \nabla u \|_{L^2(\R^3)}^{2p/(6-p)} \|u\|_{L^p(\R^3)}^{2p(5-p)/(6-p)}, \ \ \forall p \in [1,5],
\end{equation}
imply the existence of positive radial decreasing minimizers of \eqref{eq.I1a}, but only for the range $2 \leq p < 3.$

 In this work we plan to study existence and uniqueness of ground states for larger interval $2 \leq  p  < 5$ and for this reason we can define the Weinstein functional
  (see \cite{W85})
\begin{equation}
\label{WF1}
W_p(u) = \frac{\|\nabla u(x)\|_{L^2(\R^3)}^{2p/(6-p)}\|u(x)\|_{L^p(\R^3)}^{2p(5-p)/(6-p)} }{D(|u|^p,|u|^p)}
\end{equation}
and consider the associated minimization problem
\begin{equation}\label{eq.Wmin}
     W^{min}_p= \inf_{\{u \in \dot{H}^1(\R^3) \cap L^p(\R^3); u \neq 0 \}} W_{p}(u).
\end{equation}

One can easily verify the relation $ W^{min}_p = C_{GN}^{-1}$, where $C_{GN}$ is the best constant in the Gagliardo -
Nirenberg inequality \eqref{eq.I7}. Also, we can observe that the Weinstein functional $W_p(u)$ is invariant under all the symmetries
as homogeneity, scaling, translation, phase rotation and conjugation.

We will only be interested
in minimizing solutions $u$ of the functional $W_p$ that lie in the energy class $\dot{H}^1 (\R^3)\cap L^p(\R^3)$, which directly implies
that its Coulomb energy given by $D(|u|^p,|u|^p)$ is finite, thanks to the Gagliardo -
Nirenberg inequality \eqref{eq.I7}. In particular we shall establish that $u$ has
some decay rate at infinity, which becomes exponentially decaying in the classical case $p=2.$

Our first main result is the following.

\begin{thm}
\label{Te1}

Assuming $2 \leq p<5$, there is a minimizer  $u \in \dot{H}^1(\R^3) \cap L^p(\R^3)$ of $W_p,$ such that
$u$ is solution of
\begin{equation}\label{eq.EL1}
   -\Delta u + |u|^{p-2}u = I(|u|^p)|u|^{p-2}u
\end{equation}
and satisfies the Pohozaev's normalization conditions
\begin{equation}\label{eq.poh2}
   \frac{\| u\|_{L^p}^p}{5-p} =\|\nabla u\|_{L^2}^2= \frac{D(|u|^p,|u|^p)}{6-p} = k,
\end{equation}
for some $k>0$.
In addition, there exists $x_0\in \R^3,$ $z \in \C$ with $|z|=1$  and a decreasing function $Q:\R_+\to \R_+$,  so that $u(x)= z Q(|x-x_0|)$.
\end{thm}

\begin{rem}
The Pohozaev normalization conditions \eqref{eq.poh2} imply the Pohozaev identities
\begin{equation}\label{pohoz0}
\|\nabla u\|_{L^2}^2+\|u\|_{L^p}^p=D(|u|^p,|u|^p),
\end{equation}
\begin{equation}\label{pohoz}
 \|\nabla u\|_{L^2}^2  =\frac{D(|u|^p,|u|^p)}{6-p} .
\end{equation}
\end{rem}
\begin{rem}\label{AB}
If $2 < p < 5,$ then the function $Q$  from  Theorem \ref{Te1} satisfies the decay estimate
\begin{equation}\label{eq.estim1}
 Q(|x|) \leq C  |x|^{-2/(p-2)}.
 \end{equation}
 Recall that $Q$ decays exponentially in the classical case $p=2.$  For the simple proof see Remark \ref{r2}.
\end{rem}

Our second result treats the  uniqueness of minimizers $Q$ of $W_p$ satisfying \eqref{eq.poh2}, i.e.
$$ Q \in \mathcal{G} = \{u \in \dot{H}^1_{rad}\cap L_{rad}^p; W^{min}_p= \inf_{\{u \in \dot{H}^1(\R^3) \cap L^p(\R^3); u \neq 0 \}} W_{p}(u)\}$$and such that \eqref{eq.poh2} is fulfilled.

\begin{thm} \label{Th2} For any  $ 2 \leq  p < 5$ and any two radial positive minimizers $Q_1, Q_2 \in \mathcal{G},$ that satisfy \eqref{eq.poh2}, we have
$Q_1 \equiv Q_2.$
\end{thm}

It is well - known that Gagliardo - Nirenberg inequaities of type
\begin{equation}\label{eq.I7}
   D(|u|^q,|u|^q) \leq C_1 \| (\sqrt{-\Delta})^s u \|_{L^2(\R^3)}^{2q-\theta} \|u\|_{L^p(\R^3)}^{\theta}, \ \ \theta \in (0,2q),
\end{equation}
or
\begin{equation}\label{eq.I7a}
   \|u\|_{L^p(\R^3)}  \leq C_2 \| (\sqrt{-\Delta})^s u \|_{L^2(\R^3)}^{1-\theta} D(|u|^q,|u|^q)^{\theta}, \ \ \theta \in (0,1),
\end{equation}
with $ s \in (1/2,1]$
are intensively studied in the last years (see \cite{BFV14}, \cite{GS18}). In general \eqref{eq.I7a} behaves like Hardy type functional so there is no radial minimizer (see \cite{BGO16}) of
$$ \frac{\| (\sqrt{-\Delta})^s u \|_{L^2(\R^3)}^{1-\theta} D(|u|^q,|u|^q)^{\theta}}{\|u\|_{L^p(\R^3)}}$$
while positive radial minimizers of $W_p(u)$ exist and it is important problem to establish their uniqueness.
Indeed, the classical Hardy functional
$$ \frac{\|\nabla u \|^2_{L^2(\R^3)}}{\int V(x) |u(x)|^2 dx}$$
with $V(x) \sim |x|^{-2}$ has strictly positive lower bound, but there is no minimizer in $\dot{H}^1(\R^3).$ Small perturbation of $V$ such that
$$ \lim_{x \to \infty} V(x) |x|^2 = \lim_{x \to 0} V(x) |x|^2 =0$$ changes the situation (see \cite{GW18}).

Since the nonlinear terms in Gagliardo - Nirenberg estimates involve nonlocal interactions, we can not apply a Sturm comparison argument to show uniqueness of positive radial minimizers of the corresponding Weinstein functionals.

The classical case $p=q=2, s=1$  have been studied in \cite{Lieb}, where the uniqueness approach is based on shooting method and the fact that the asymptotic behavior of the Riesz potential is
\begin{equation}\label{eq.tr1}
   I(|u|^2)(x) = \frac{\|u\|^2_{L^2}}{4\pi |x|} + o\left(|x|^{-1} \right), \ \ x \to \infty.
\end{equation}
In this case the Pohozhaev normalization conditions \eqref{eq.poh2} become
$$
   \frac{\| u\|^2}{3} = \|\nabla u\|^2 = \frac{D(|u|^2,|u|^2)}{4} = k.
$$
Indeed, taking any two solutions $u_1,u_2$, we use the normalization conditions and from \eqref{eq.tr1} we can deduce that
$$ I(|u_1|^2)(x) - I(|u_2|^2)(x) = o\left(|x|^{-1} \right), \ \ x \to \infty.$$
This fact gives the possibility to apply Sturm argument and by following shooting method to deduce  uniqueness.

If $p\neq 2, $
then from Remark \ref{AB} we face the first obstacle, namely the loss of the exponential behavior at infinity. Nevertheless, the decay rate given by \eqref{eq.estim1} (see also Lemma \ref{l.de3}) is sufficient to get
\begin{equation}\label{eq.tr12}
  I(|Q|^p)(x) = \frac{\|Q\|_{L^p}^p}{4\pi |x|} + o\left(|x|^{-1} \right), \ \ x \to \infty
\end{equation}
and obviously we gain control on the asymptotics of Riesz potential at infinity, since the $L^p$ norm is a conserved Pohozhaev quantity.
Applying then a Gronwall argument (see Lemmas \ref{lgm} and \ref{l.g1}), we can conclude that the function
$$ \varphi(r) = |Q_1(r) - Q_2(r)| + |I(|Q_1|^p)(r) - I(|Q_2|^p)(r)| $$is identically zero for $r\in [0;\infty)$.

The outline of this paper is as follows. In Section 2, we briefly describe the basic steps in proving existence of positive, radial and decreasing minimizers and then we focus on establishing the asymptotic behavior of $Q$ and $I(|Q|^p)$ . In Section 3, we apply variational techniques  to get the Euler - Lagrange equation \eqref{eq.EL1} and Pohozhaev normalization conditions \eqref{eq.poh2}. Then, exploiting the asymptotic behavior from Section 2 and using a modified Gronwall Lemma we prove  the
uniqueness result of Theorem \ref{Th2}. A short proof of the Gronwall type Lemma is given in the Appendix.

\section{Far field decay estimates}

The proof of the existence of radial, positive and decreasing minimizer  $u \in \dot{H}^1(\R^3) \cap L^p(\R^3)$ for $p \in (2,5)$ is a standard argument, based on Gagliardo - Nirenberg inequality \eqref{eq.I7}, the Strauss Lemma \ref{l.S1} and Schwartz symmetrization (see, for instance \cite{Lieb1}).

Here, we can refer to the Diamagnetic inequality implying (see \cite{LL}) that $$\|\nabla |u| \|_{L^2}\leq \|\nabla u\|_{L^2}$$and therefore, our minimizer is nonnegative.
Even in the nonlocal case, the problem that the minimizers  of
$W_p(u)$ are radially symmetric functions, is easy to be
proved. A classical approach to radial symmetry of minimizers is
Schwarz symmetrization (or spherical decreasing rearrangement
\cite{LL}). For a nonnegative function $u $, its
symmetrization $u^*$ is a radially-decreasing function from $\R^n$
into $\R$, which has the property that for any $a > 0$
\begin{equation}\label{eq.rearrang}
\mu\{x \in \R^n:  u^*(x)
> a \} = \mu\{x \in \R^n:  u(x) > a \},
\end{equation}
where $\mu$ is Lebesgue measure. It is well-known  that $u^*$
satisfies the inequalities
\begin{equation}\label{eq.rear-ineq1}
\| \nabla u^*\|^2_{L^2} \leq \| \nabla u\|^2_{L^2},
\end{equation}
for any  $u \in \dot{H}^1(\R^n)$ and
\begin{equation}\label{eq.rear-ineq2}
\| u^*\|_{L^p} =\| u\|_{L^p},
\end{equation}
for all $1\leq p<\infty$ such that $u \in L^p(\R^n)$. Moreover, for any nonnegative $u \in L^p(\R^n)$ we have
\begin{eqnarray}\label{eq.rear-ineq3}
D({u^*}^p,{u^*}^p) \geq D({u}^p,{u}^p),
\end{eqnarray}
due to the Riesz inequality
for rearrangements (see, for instance Lemma 3 in \cite{Lieb}). Finally,  the existence of a
sequence of radially symmetric minimizers can be proved, following the idea of Lions \cite{PLLions} and using some important properties of the Coulomb functional $D(\cdot,\cdot)$, established in \cite{GV}.

Therefore, from now on we shall assume that $$u \in \dot{H}^1(\R^3) \cap L^p(\R^3)$$
is radially symmetric, positive and decreasing
solution of
\eqref{eq.EL1}. Thus, equation \eqref{eq.EL1} can be transformed into the following system of ordinary differential equations
 \begin{align}\label{eq.Sb2c}
&- r^{-2}\left( r^2u^\prime(r) \right)^\prime +  u^{p-1}(r) =
A(r) u^{p - 1}(r),
\nonumber
\\
& - r^{-2}\left( r^2A^\prime(r) \right)^\prime =
 u^{p }(r),
  \end{align}
where $$ A(|x|) := I(u^p)(|x|) .$$ An application of Newton's theorem for radially symmetric functions
(see Theorem 9.7 in \cite{LL}) to the solution $A$ of Poisson equation in \eqref{eq.Sb2c} implies the representation formula
\begin{eqnarray}\label{eq.Sb2a}
 \nonumber   A(|x|)=\int_0^{\infty}\frac{u^p(s)s^2 ds}{\max\{|x|, s\}} \\
     = \frac {1}{|x|}\|u\|_{L^p}^p+\int_{|x|}^\infty\left(\frac{1 }{s}-\frac{1
}{|x|}\right)u^p(s) s^2ds.
  \end{eqnarray}
  We have a variant of Strauss lemma (see \cite{St77}).
  \begin{lem} \label{l.S1}
  If $u \in \dot{H}_{rad}^1(\R^3) \cap L^p_{rad}(\R^3)$ is a function which decays sufficiently rapidly at $\infty$, then we have the estimate
  \begin{equation}\label{eq.ST1}
    r^{4/(p+2)}u(r) \leq C \|\nabla u\|_{L^2}^{2/(p+2)} \|u\|_{L^p}^{p/(p+2)}.
  \end{equation}
  \end{lem}
  \begin{proof} Without loss of generality we can take $u$ positive.
  We have the relation
  $$ r^2 u^{(p+2)/2}(r) = 2 \int_0^r  u^{(p+2)/2}(s) s ds + \frac{p+2}{2} \int_0^ru^\prime(s) u^{p/2}(s)  s^2 ds.$$
  Applying Cauchy's and Hardy's inequalities we get
  $$  r^2 u^{(p+2)/2}(r) \leq C \|\nabla u\|_{L^2} \|u\|_{L^p}^{p/2}$$
  and this completes the proof.

  \end{proof}

  Unfortunately, the Strauss type estimate \eqref{eq.ST1} is not sufficient to check the key property
  \begin{equation}\label{eq.Sb2d}
    \lim_{r \to \infty} r A(r) = \|u\|_{L^p}^p  .
  \end{equation}

However, for $p >2 $ we can apply \eqref{eq.ST1} and via the representation formula \eqref{eq.Sb2a} to get
\begin{equation}\label{eq.Sb3}
    A(r) \leq \frac{C}{r^\delta}, \ \delta= \frac{2(p-2)}{p+2} >0  .
  \end{equation}

We shall need a stronger decay estimate and therefore we pose the following Lemma.
\begin{lem} \label{l.de3}
 If  $p \in (2,5)$ and $u \in \dot{H}_{rad}^1(\R^3) \cap L^p_{rad}(\R^3)$ is  a positive decreasing solution of \eqref{eq.EL1}, then for any $r_0$ sufficiently large we have the estimate
  \begin{equation}\label{eq.ST10}
    u(r) \leq \frac{C}{r^{2/(p-2)}} , \ \ r > r_0.
  \end{equation}
\end{lem}

\begin{proof}
The weak decay estimate \eqref{eq.Sb3} and the first equation in \eqref{eq.Sb2c} show that
$$ A(r) < \frac{1}{2}, \ \ \forall r > r_0 $$ and $r_0$ sufficiently large. Therefore we can assert that
$$ -\Delta u (r) + \frac{u^{p-1}(r)}{2} := F(r) \leq 0, \ \ \forall r > r_0.$$
Setting
$$ u_0(r) = \frac{C}{r^{2/(p-2)}}$$
and choosing $C$ sufficiently large, we can verify the property
$$ -\Delta u_0 (r) + \frac{u_0^{p-1}(r)}{2} := F_0(r) > 0, \ \ \forall r > r_0.$$
Now, we are in position to apply the maximum principle and deduce
\begin{equation}\label{eq.as1}
    u(r) \leq u_0(r) , \ \ r > r_0.
\end{equation}
Indeed, if $I=(r_1,r_2)$ is an interval in $(r_0,\infty)$, such that
$$ u(r) > u_0(r), \ \ r \in I, \ \ u(r)=u_0(r), \ \  r \in \partial I,$$
then the minimum of $(u_0-u)$ is on the boundary $\partial  I$ of $I,$ since the minimum in interior point $r^* \in I$ means $ \Delta (u_0-u)(r^*)  \geq 0,$
while our choice of $u_0$ implies
$$ \Delta (u_0 - u)(r) = \frac{1}2 (u_0^{p-1}(r) - u^{p-1}(r)) -F_0(r) + F(r) < 0 , \ \ \forall r \in I.$$
The contradiction shows that \eqref{eq.as1} is fulfilled and the proof is finished.
\end{proof}

\begin{rem} The above estimate and the assumption $p \in (2,5) $ imply now \eqref{eq.Sb2d}.
\end{rem}
\begin{rem} \label{r2} If $p=2$, then the identity \eqref{eq.ST1} becomes
$$  -\Delta u (r) + \frac{u(r)}{2} := F(r) \leq 0, \ \ \forall r > r_0,$$
so we can take
$$ u_0(r) = \frac{e^{-\varepsilon r}}r $$
with $\varepsilon \in (0,1/2).$
\end{rem}

\begin{rem} \label{r3} A natural question is if the decay  estimate \eqref{eq.ST10} of Lemma \ref{l.de3} is optimal. One can use the upper bound  \eqref{eq.ST10} and via
$$ A(|x|) := I(u^p)(|x|) \lesssim |x|^{-4/(p-2)}$$
it is possible to evoke   a comparison arguments (see Lemma 2.15 in \cite{AM}) and to deduce appropriate  lower bound of $u.$
\end{rem}

\section{Uniqueness of the minimizers of Weinstein's functional}

Our first result in this section is to derive the Pohozaev identities of Theorem \ref{Te1}. Since the classical case $p=2$ is a well-known result due to Lieb  \cite{Lieb}, then from now on we can assume $p \in (2,5).$
So, we can turn to the derivation of Euler - Lagrange equation \eqref{eq.EL1}.

For the first variation of
\begin{equation}\label{eq.var1}
 W_p(Q+\varepsilon h),
\end{equation}
we take for simplicity a positive real-valued function $h\in C^\infty_0(\R^3)$  and $\varepsilon>0.$ Then we have
 the relations
\begin{align}\label{eq.mrI}
    \frac d {d\varepsilon}\left( W_p(Q+\varepsilon h)\right)
    =W_p(Q+\varepsilon h) R(\varepsilon),
\end{align}
where
\begin{align}\label{eq.2var}
R(\varepsilon) =
\frac {2p \, \langle \nabla Q+\varepsilon \nabla h, \nabla h\rangle_{L^2}}{(6-p)\left\| \nabla Q+\varepsilon \nabla h\right\|^{2}_{L^{2}}}
+\frac{2p(5-p)\, \langle (Q+\varepsilon h)^{p-1}, h\rangle_{L^2}}{(6-p)\left\| Q+\varepsilon h\right\|_{L^p}^{p}}&
\nonumber\\
-\frac{2p D((Q+\varepsilon h)^{p-1} h,(Q+\varepsilon h)^{p})}{D((Q+\varepsilon h)^p,(Q+\varepsilon h)^p)}.
\end{align}
Taking $\varepsilon =0$ and using the fact that $Q$ is a minimizer of $W_p$, we get the equation
$$ -\frac {\, \Delta Q}{(6-p)\left\| \nabla Q\right\|^{2}_{L^{2}}}
+\frac{(5-p)\, Q^{p-1}}{(6-p)\left\| Q\right\|_{L^p}^{p}} =\frac{ I(Q^{p}) \, Q^{p-1}}{D(Q^p,Q^p)}.$$

Using now a rescaling argument, we can assume the Pohozaev normalization conditions \eqref{eq.poh2} fulfilled and then the equation becomes  \eqref{eq.EL1}. In fact, this completes the proof of Theorem \ref{Te1}.
\\

Our next step is to prove the uniqueness of the Weinstein minimizers.
First, we recall Lemma \ref{l.de3} and following its proof, we get a more precise asymptotics of $Q$ and $A=I(Q^p),$ namely
\begin{equation}\label{eq.as3}
  Q(r) = \frac{C_*}{r^{2/(p-2)}} \left( 1 + o(1) \right), \ A(r) = \frac{\|Q\|_{L^p}^p}{r} \left( 1 + o(1) \right),
\end{equation}
for $r \to \infty.$ Here and below $C_*>0$ is a universal constant, determined by the equation
and space dimension.

The key idea to obtain the uniqueness result is to assume that $Q_1$ and $Q_2$ are two radially symmetric, positive and decreasing solutions
$$ Q_1(r), Q_2(r) \in \dot{H}^1_{rad}(\R^3) \cap L^p_{rad} (\R^3)$$ of the differential equation
\begin{equation}\label{eq.un1}
- Q^{\prime\prime}_j(r) - \frac{2Q_j^\prime(r)}{r}  - A_j(r) \ Q_j(r)^{p - 1}  +    Q_j(r)^{p-1} =0,\ \ j=1,2,
\end{equation}
with
$$ A_j(r) = I(Q^p_j)(r) = \int_0^{\infty}\frac{Q_j(s)^p s^2 ds}{\max\{r,s\}},$$
such that the asymptotic expansions \eqref{eq.as3} are fulfilled, i.e.
\begin{equation}\label{eq.un2}
  Q_j(r) = \frac{C_*}{r^{2/(p-2)}} \left( 1 + o(1) \right), \ A_j(r) = \frac{\|Q_j\|_{L^p}^p}{r} \left( 1 + o(1) \right)
\end{equation}
and to show that
\begin{equation}\label{eq.un3}
  \|Q_1\|_{L^p} = \|Q_2\|_{L^p}.
\end{equation}
But all these  follow directly from Pohozaev normalization conditions \eqref{eq.poh2}.

To complete the proof of the uniqueness, we shall consider the function
\begin{equation}\label{eq.phi-def}
   \varphi(r) = |Q_1(r) - Q_2(r)| + |A_1(r) - A_2(r)|,
\end{equation}
in the interval
$ r \in (r_0, \infty),$ where $r_0>$ is sufficiently large.

The estimate of
$|Q_1(r) - Q_2(r)|$
uses in an essential way the fact that the leading terms of the asymptotic expansions for $Q_j(r), j=1,2$ coincide,
so
\begin{equation}\label{Gr3m}
 Q_j(r) = \frac{C}{r^{2/(p-2)}} \left( 1 + w_j(r) \right), \ \ w_j(r) = O\left(r^{-\varepsilon}\right),
\end{equation}
due to \eqref{eq.un2}.
Setting
$$ u_0(r) = \frac{C}{r^{2/(p-2)}},$$ we have
$$ Q_j(r)^{p-1} = u_0^{p-1} F(w_j(r)), \ \ F(x) = (1+x)^{p-1}.$$

This relation suggests the idea to treat the troublesome term $Q_j^{p-1}$ in the equation \eqref{eq.un1} satisfied by $Q_j.$

We have the following cancelation property needed in the sequel
\begin{eqnarray}\label{eqq.Gr4m}
    Q_1^{p-1} - Q_2^{p-1} = \\ \nonumber =(p-1) u_0^{p-2}\left( Q_1 - Q_2 \right) + O\left(u_0^{p-2}|Q_1-Q_2|(|w_1|^{p-2}+|w_2|^{p-2}) \right),
\end{eqnarray}
due to the property
$$  F(x_1) - F(x_2)= F^\prime(0) (x_1-x_2) + O\left( |x_1-x_2|(|x_1|^{p-2}+|x_2|^{p-2})\right)$$
for $p \in (2,5)$ and for all $x_1, x_2 \in \R$ close to zero.

Note that $U=Q_1-Q_2$ is a solution to the problem
\begin{equation}\label{eq.FU3}
   -\Delta U + (p-1) u_0^{p-2} U = G(u_1,u_2),
\end{equation}
where
$$ G(Q_1,Q_2)=  I(Q_1,Q_2)+ II(Q_1,Q_2) $$
with
$$  I(Q_1,Q_2) = -Q_1^{p-1} + Q_2^{p-1} + (p-1) u_0^{p-2}\left( Q_1 - Q_2 \right), $$
$$ II(Q_1,Q_2)= A_1 Q_1^{p-1}- A_2 Q_2^{p-1} .$$

Now we apply the maximum principle and deduce
$$ |U(r)| \leq U^*(r), $$
where
$U^*(r)$ is a solution to the problem
 \begin{equation}\label{eq.FU4}
   -\Delta U^*  = |G(Q_1,Q_2)|,
\end{equation}
satisfying Dirichlet boundary condition
$$ U^*(r_0) = C_1 r_0^{-2/(p-2)},$$
with $C_1 >0$ sufficiently large.
Applying a priori estimate for the solution of \eqref{eq.FU4}, we get the estimate
$$
  |U(r)| \leq   U^*(r) \leq \frac{C}{r}\ \int_r^\infty  |G(Q_1(s),Q_2(s))| s^2 ds.
$$

Since
$$  G(Q_1(s),Q_2(s)) = I(Q_1,Q_2)(s) + II(Q_1,Q_2)(s),$$
we can estimate the term $I(s)=I(Q_1(s),Q_2(s))$ as follows
$$ |I(s)| \leq \frac{C \varphi(s)}{s^{2+ \varepsilon}},$$
due to \eqref{eqq.Gr4m} and the bound \eqref{Gr3m} for $w_j$, while the second term satisfies
$$ |II(s)| \leq \frac{C \varphi(s)}{s^{3}}.$$
In conclusion, we get
$$
   |U(r)| \leq \frac{C}{r}\ \int_r^\infty \frac{ \varphi(s)ds}{s^{\varepsilon}}
$$
for all $r > r_0.$

To estimate $|A_2(r)-A_1(r)|$ we can write
\begin{equation}\label{eq.Gr2m}
    |A_2(r)-A_1(r)| \leq \frac{C}{r}\ \int_r^\infty \frac{ \varphi(s)s^2 ds}{s^{2(p-1)/(p-2)}}.
\end{equation}

The above estimate and \eqref{eq.Gr2m} imply
$$
   \varphi(r) \leq \frac{C}{r}\ \int_r^\infty \frac{ \varphi(s)ds}{s^{\varepsilon}},
$$
and thus we arrive at the following assertion.

\begin{lem} \label{lgm} There exists a constant $C >0,$ so that the function $\varphi$, defined in \eqref{eq.phi-def}, satisfies
$$ \varphi(r) \leq \frac{C}{r}  \int_r^\infty \frac{\varphi(s) ds}{s^{\varepsilon}},  $$
for any $r> R$ and $0 < \varepsilon < 2/(p-2).$
\end{lem}

\begin{proof}[Proof of the uniqueness of ground state]

Applying the Gronwall argument of Lemma \ref{l.g1}, we use the estimate of Lemma \ref{lgm}  and conclude that
$$ \varphi(r) = |Q_1(r) - Q_2(r)| + |A_1(r) - A_2(r)| $$ is identically zero for
$ r >r_0.$ Since $ (Q_j,A_j), j=1,2$ are solutions to the Cauchy problem
\begin{alignat}{1}\label{eq.ODE-sysm}
&Q_j^{\prime \prime}+\frac{2}{r}\ Q_j^{\prime} =   Q_j^{p - 1}(1-A_j), \\
\nonumber & A_j^{\prime \prime}+\frac{2}{r}\ A_j^{\prime}  =  -Q_j^{p },
\end{alignat}
with
$$ Q_1(r)-Q_2(r)=A_1(r)-A_2(r) = 0, \ \ \forall r > r_0,$$
we obtain
$$Q_1(r)-Q_2(r)=0, \ \ \forall r > 0.$$

This implies that positive, radial, decreasing  ground state satisfying Pohozhaev conditions is unique and completes the proof.
\end{proof}

\section{Appendix: Gronwall type lemma}

\begin{lem} \label{l.g1}
If $\varepsilon > 0,$ $\psi(r) \in C(1,\infty)$ is a nonnegative function satisfying
\begin{equation}\label{eq.A1}
    \psi(r) \leq C, \ \ \forall r > 1,
\end{equation}
and
$$ \psi(r) \leq C \int_r^\infty \frac{\psi(s) ds}{s^{1+\varepsilon}}, \ \ \forall r> 1,$$
then $\psi(r) =0$ for $r >1.$
\end{lem}
\begin{proof}
The assumption \eqref{eq.A1} and the integral estimate for $\psi$ imply
$$ \psi(r) \leq \frac{C}{r^{\varepsilon/2}}  \to 0 \ \mbox{as $r \to \infty$.} $$ Making the change of variable
$ r \to \rho=1/r,$ we reduce the proof to the following statement for
$$ \Psi(\rho) = \psi \left( \frac{1}{\rho}\right).$$

If $\Psi$ is a continuous non-negative function on $I=[0,R]$ such that
$ \Psi(0) = 0$ and
$$ \Psi(\rho) \leq \int_0^\rho \frac{\Psi(\sigma) d\sigma}{\sigma^{1-\varepsilon}}, \ \forall \rho \in I,$$
then $\Psi(\rho) \equiv 0$ in $I.$ This is the standard Gronwall lemma and the proof is now completed.
\end{proof}



\begin{thebibliography}{9}

\bibitem{BFV14} J. Bellazzini, R. L. Frank and N. Visciglia,  { Maximizers for Gagliardo -- Nirenberg inequalities
and related non-local problems}, \textit{Math. Ann.} vol. 360 (2014), 653 -- 673; doi: http://dx.doi.org/10.1007/s00208-014-1046-2.

\bibitem{BGO16} J. Bellazzini, M.  Ghimenti and T. Ozawa,{ Sharp lower bounds for Coulomb energy,} \textit{ Math. Res. Lett.} 23 (2016), no. 3, 621 -- 632; doi.org/10.4310/MRL.2016.v23.n3.a2.
    
\bibitem{AM} L. D' Ambrosio and E. Mitidieri, Nonnegative solutions of some
quasilinear elliptic inequalities and applications,  \textit{ Mat. Sb.} 201, no. 6, (2010), 75 -- 92;
doi: https://doi.org/10.4213/sm7585

\bibitem{GV} H. Genev and G. Venkov,  { Some properties of Coulomb energy space}, {\em AIP Conference Proceedings}, vol. 1497, 290 (2012); doi: http://dx.doi.org/10.1063/1.4766796.

\bibitem{GS18} V.Georgiev, A. Stefanov,
{ On the classification of the spectrally stable standing waves of the Hartree problem }, \textit{Physica D: Nonlinear Phenomena} (available online, January , 2018) doi: 10.1016/j.physd.2018.01.002

\bibitem{GW18} V.Georgiev, J.Wirth, {Zero resonances for localised potentials},  (2018) arXiv:1803.01813


\bibitem{Lieb} E.H. Lieb,  { Existence and uniqueness of the minimizing solution of Choquard's nonlinear equation,} {\em  Stud. Appl.
Math.}  57(2)  (1977), 93--105.

\bibitem{Lieb1} E.H. Lieb, { Sharp constants in the Hardy-Littlewood-Sobolev and related inequalities}, {\em Ann.
of Math.} 118 (1983), 349--374.

\bibitem{LL} E.H. Lieb and M. Loss,  {Analysis}, Second Edition, Graduate Studies in Mathematics, vol. 14, AMS, Providence, RI, 2001.

\bibitem{PLLions} P.-L. Lions, {  The Choquard equation and related questions,} {\em  Nonlinear Anal. } 4(6) (1980),  1063--1072.



\bibitem{St77} W. ~Strauss,{ Existence of Solitary Waves in Higher Dimension}, \textit{Comm. Math. Physics}, Vol.~55 (1977),  149--162.



\bibitem{W85} M. Weinstein, { Modulational stability of ground states of nonlinear Schr\"odinger equations}, SIAM J. Math. Anal, Vol. 16(3) (1985), 472-–491.


\end{thebibliography}
\end{document}